\documentclass[final]{amsart}

\usepackage{showkeys}
\usepackage{xypic}
\usepackage{amssymb}
\usepackage{amscd}
\usepackage[mathscr]{eucal}

\renewcommand{\Pr}{\ensuremath{\mathbb P}}

\newcommand{\coh}[2]{{\rm H}^{#1}(#2)}

\newcommand{\isoarrow}{\tilde{\longrightarrow}}
\newcommand{\gr}[2]{{\rm Gr}^{#1}_{#2}}

\newcommand{\struct}[1]{{\mathcal O}_{#1}}

\newcommand{\spec}{{\rm Spec}}

\renewcommand{\hom}[2]{{\rm Hom}_{#1}(#2)}

\newcommand{\gm}{{\mathbb G}_m}

\newtheorem{lemma}{Lemma}[section]
\newtheorem{proposition}[lemma]{Proposition}
\newtheorem{theorem}[lemma]{Theorem}
\newtheorem{corollary}[lemma]{Corollary}
\newtheorem{claim}[lemma]{Claim}
\newtheorem*{utheorem}{Theorem}

\theoremstyle{definition}
\newtheorem{definition}[lemma]{Definition}

\theoremstyle{remark}
\newtheorem{example}[lemma]{Example}
\newtheorem{remark}[lemma]{Remark}
\newtheorem{notation}[lemma]{Notation}
\newtheorem{construction}[lemma]{Construction}

\newcommand{\bbmu}{\mu}

\newcommand{\et}{{\rm {\acute{e}}t}}
\newcommand{\ed}{{\rm ed}}
\newcommand{\e}{{\rm e}}
\newcommand{\fk}{{\rm Fields}_{\Bbbk}}
\newcommand{\sets}{{\rm Sets}}
\newcommand{\ri}{{\rightarrow}}
\renewcommand{\gr}{{\rm gr}}

\newcommand{\F}{{\mathscr{F}}}
\renewcommand{\cL}{{\mathscr L}}
\newcommand{\E}{{\mathscr E}}
\newcommand{\K}{{\mathscr K}}
\newcommand{\G}{{\mathscr G}}

\newcommand{\aff}{{\rm Aff}}
\newcommand{\sX}{{\mathfrak X}}
\newcommand{\sG}{{\mathfrak G}}

\newcommand{\bun}{{\rm Bun}}
\renewcommand{\sl}{{\rm SL}}
\newcommand{\gl}{{\rm GL}}
\newcommand{\sln}{{\sl_n}}
\newcommand{\gln}{{{\rm GL}_n}}
\newcommand{\glq}{{{\rm GL}_h}}
\newcommand{\su}{{\rm SU}}
\newcommand{\univ}{{\rm univ}}
\newcommand{\isomF}{{\mathcal Isom}}
\newcommand{\isom}{{\textbf{Isom}}}

\newcommand{\Ext}{\mathrm{Ext}}

\newcommand{\id}{\mathrm{id}}

\newcommand{\Ker}{\mathrm{Ker}}
\newcommand{\Stab}{\mathrm{Stab}}

\newcommand{\rk}{\mathrm{rank}}
\newcommand{\trdeg}{\mathrm{trdeg}}
\begin{document}

\title[moduli of ${\sl_n}$-bundles]{Upper bounds for  the essential dimension of the moduli stack of ${\sl_{n}}$-bundles
over a curve}

\author{Ajneet Dhillon and Nicole Lemire}

\begin{abstract}
 We find upper bounds for the essential dimension of
various moduli stacks of  $\sln$-bundles over a curve. 
When $n$ is a prime power, our calculation computes the essential dimension
of the stack of stable bundles exactly and the essential dimension is
not equal to the dimension in this case.
\end{abstract}

\maketitle

\section{Introduction}

We work over a field $\Bbbk$ of characteristic $0$ and fix a smooth projective 
geometrically connected
curve $X$ of genus $g\ge 2$ over $\Bbbk$. We assume that $X$ has a point over $\Bbbk$.
Our purpose in this paper is to study the essential dimension of various moduli stacks
of $\sln$-bundles on our curve. In order to use inductive arguments on the rank
it will be convenient to slightly generalize the question. Let $\xi$ be a line bundle 
on our curve. We will study the essential dimension of the stacks
\[
 \bun^{s,\xi}_\sln\quad \bun^{ss,\xi}_\sln \quad \bun^{\xi}_\sln
\]
of (resp. stable, semistable, full) bundles with an identification of the top 
exterior power
with $\xi$.

If our stacks possessed fine moduli spaces the essential dimension would just be
the dimension of the moduli space. As no such space exists the question is open.
At least when $\gcd(n,\xi)$ is a prime power, it seems that
the essential dimension does
not agree with the dimension of the  moduli stack in the stable case.

For the stable case we compare the stack with its moduli space and use 
some theorems in \cite{brosnan:07} to study the essential dimension of the moduli 
stack. To pass from stable to semistable we use the Jordan-H\"older
filtration. Some care is needed here as when considering essential dimension, 
one is forced into a position of having to consider non-algebraically closed 
fields even if the base field $\Bbbk$ is algebraically closed. For a semistable
bundle its Jordan-H\"older filtration may not be defined over the base field
if it is not algebraically closed. To pass to the full moduli stack, we use the
Harder-Narasimhan filtration.

An outline of the paper follows. Section 2 contains a review of the
notion of essential dimension and pertinent results. Section 3 contains
a review of the notions of stable and semistable with a view towards curves
over non-algebraically closed fields. Section 4 proves some elementary 
properties of our moduli stacks that will be needed later. Section 5 reviews
twisted sheaves on gerbes and their relationship with period and index. Section
6 lists results regarding the Brauer group of the moduli space of vector bundles.
A key invariant that is needed in our computations is the generic index of the gerbe
\[
 \bun^{s,\xi}_\sln\ri \su(X,n)^s,
\]
where $\su(X,n)^s$ is the coarse moduli space. In \cite{balaji:07} and \cite{drezet:89} the
period of this gerbe is studied. In Section 7 we observe that the existence of 
some natural twisted sheaves implies that period equals index for this gerbe.
 The bound for the 
essential dimension of the stack $\bun^{s,\xi}_\sln$ is obtained in section 8.
When $\gcd(\deg \xi, n)$ is a prime power this bound is an equality.
The remaining sections contain results on bounds for the essential dimension
of the full moduli stack and the semistable locus. 

To describe the final result, we introduce a function
$h_g:{\mathbb N}\rightarrow {\mathbb N}$ defined recursively by
\begin{eqnarray*}
 h_g(1) &=& 1 \\
 h_g(n) - h_g(n-1) &=& 
 (n^3-n^2)  +  \frac{n^2}{4}(g-1)  + \frac{n}{2} + \frac{n^2 g^2}{4} + \frac{1}{4}
\end{eqnarray*}
The final result 
(Theorem 11.1) that we obtain is 
the following:

\begin{utheorem}
We have
$$
\ed(\bun^\xi_\sln) \le \lfloor h_g(n)\rfloor + 1.
$$
\end{utheorem}

\section*{Acknowledgments}
The authors would like to thank Patrick Brosnan, Daniel Krashen and Zinovy Reichstein
for useful discussions and ideas. We would also like to thank the referees for
numerous corrections to our original manuscript.

\section*{Notation and Conventions}

\begin{itemize}
 \item $\Bbbk$ our base field of characteristic $0$.
 \item $X$ a smooth geometrically connected curve of genus $\ge 2$ defined over $\Bbbk$
and having a point over $\Bbbk$.
 \item $\bun^{\xi}_\sln$ the moduli stack of  bundles over our curve with a fixed isomorphism of the top exterior
power with $\xi$.
 \item $\bun^{\xi}_\gln$ the moduli stack of bundles over our curve having determinant $\xi$.
 \item $\bun^{\xi,s}_\gln,\  \bun^{\xi,s}_\sln$ the open substacks of  stable bundles.
 \item $\bun^{\xi,ss}_\gln,\  \bun^{\xi,ss}_\sln$ the open substacks of semistable bundles.
 \item $\su(X,\xi)^s$ the moduli space of stable vector bundles with determinant $\xi$. 
\end{itemize}

\section{Essential Dimension}

We denote by $\fk$ the category of field extensions of
$\Bbbk$. Let $F:\fk\ri\sets$ be a functor. We say that
$a\in F(L)$ is \emph{defined over a field} $K\subseteq L$ if
there exists a $b\in F(K)$ so that $r(b)=a$ where $r$ is the
restriction
\[
F(K)\ri F(L).
\]
The \emph{essential dimension} of $a$ is defined to be 
$$
\ed(a)\stackrel{\rm def}{=}{\rm min}_K {\rm tr.deg}_\Bbbk K,
$$
where the minimum is taken over all fields of definition $K$ of
$a$.

The \emph{essential dimension} of $F$ is defined to be
$$
\ed(F) = {\rm sup}_a \ed(a),
$$
where the supremum is taken over all $a\in F(K)$ and $K$ varies
over all objects of $\fk$.

For an algebraic stack
$\sX\rightarrow \aff_\Bbbk$
we obtain a functor
$$ \fk\ri \sets,$$
which sends $K$ to the set of isomorphism classes
of objects in $\sX(K)$. We define 
\emph{ the essential dimension of } $\sX$ to be the
essential dimension of this functor, and denote this
number by $\ed_{\Bbbk}(\sX)$.

We now recall some theorems from
\cite{brosnan:07} that will be needed in the future. We assume for the
remainder of this section that $\sX/\Bbbk $ is a Deligne-Mumford stack,
locally of finite type,
with finite inertia. By, \cite{keel:97}, such a stack has a coarse moduli 
space $M$. The first result that we shall need is

\begin{theorem}\label{t:main}
Suppose that ${\rm char}(\Bbbk)=0$ and $\sX$ is also
smooth and connected. Let $K$ be the field of rational
functions on $M$ and let $\sX_K = \spec(K) \times_{\spec(K)} \sX$
be the base change. Then
$$
\ed_{\Bbbk}(\sX) = \dim M + \ed_K(\sX_K).
$$
\end{theorem}

The stack $\sX_K/K$ is called the generic gerbe. In the case
where this gerbe is banded by $\bbmu_n$, more can be said about
$\ed_K(\sX_K)$. 

Let $\sG$ be a gerbe over our field $\Bbbk$ banded
by $\bbmu_n$. Such a gerbe gives a torsion class in the Brauer group ${\rm Br}(K)$.
The index of this class is called the \emph{index} of the gerbe and denoted by
${\rm ind}(\sG)=d$. There is a Brauer-Severi variety $P/\Bbbk$ of dimension $d-1$
whose class maps to the class of $\sG$ via the connecting homomorphism
$$
\coh{1}{X,{\rm PGL}_d}\ri\coh{2}{X,\gm}.
$$

Let $X$ be a smooth and proper variety over $\Bbbk$.
The set $X(\Bbbk(X))$ is the collection of 
rational endomorphisms of $X$ defined over $\Bbbk$.
Define
$$ \e_\Bbbk(X) = \inf\{\dim \overline{{\rm im}(\phi)} \mid
\phi\in X(\Bbbk(X)) \}.$$
The number $\e_\Bbbk(X) $ is called the \emph{canonical dimension} of $X$.

\begin{theorem}\label{t:main2}
In the above situation
$$
\ed(\sG) = \e_K(P) + 1.
$$
\end{theorem}

\begin{proof}
See \cite[Theorem 7.1]{brosnan:07}.
\end{proof}

\begin{corollary}
In the above situation if $n=p^r$ is a prime power we have
$$
\ed(\sG) = {\rm ind}(P) + 1.
$$
\end{corollary}

\begin{proof}
See \cite[Theorem 2.1]{kar:00} and \cite{merk:03}.
\end{proof}

In this paper we will be interested in studying the essential dimension
of the stack $\bun^{\xi}_\sln$. Let us recall what it is precisely.

Fix a line bundle $\xi$ on our curve $X$ and denote
by $\bun^{\xi}_\sln$ the moduli stack of 
$\sln$-vector bundles on $X$ with determinant $\xi$. For a 
$\Bbbk$-scheme $U$ the objects in the groupoid over $U$
are pairs $(\E,\phi)$ where $\E$ is a rank $n$  bundle
on $X\times_\Bbbk U$
and $\phi$ is an isomorphism
$$
\phi:\bigwedge^n\E\isoarrow\xi.
$$
A morphism $(\E,\phi)\rightarrow (\E',\phi')$ is an isomorphism
of vector bundles $\alpha:\E\isoarrow\E'$ such that the following
diagram commutes:
$$
\xymatrix{
\bigwedge^n\E \ar[r]^\phi \ar[d]_{\wedge\alpha} & \xi \ar@{=}[d]\\
\bigwedge^n\E' \ar[r]^{\phi'} & \xi.
}
$$
In the case where $\xi$ is the trivial bundle this is just
the moduli stack of  $\sln$-torsors.

In order to study the essential dimension of this stack it will
be useful to introduce another auxiliary stack $\bun^{\xi}_\gln$.
For a 
$\Bbbk$-scheme $U$ the objects in the groupoid over $U$ are
rank $n$ vector bundles 
on $X\times_\Bbbk U$ with $\det\E \otimes\text{pr}_X^* \xi^\vee$
isomorphic to $\text{pr}_U^* \eta$ where $\eta$ is a line bundle on
$U$. The morphisms of the groupoid are just isomorphisms of
vector bundles. It follows from the generalized seesaw theorem,
\cite[pg. 89]{mumford:70} that this is in fact a closed substack
of the moduli stack of vector bundles on $X$.

\section{Stability and semi-stability for bundles}

\begin{notation}
Let $\E$ be a vector bundle on $X_K$. We denote by 
$\E_L$ the pullback of $\E$ under the natural projection
$X_L\rightarrow X_K$ where $K\hookrightarrow L$ is a 
field extension.
\end{notation}

Let $\E$ be a vector bundle on our curve $X$. The \emph{slope}
of $\E$ is defined to be
$$
\mu(\E)\stackrel{{\rm defn}}{=} \frac{\deg(\E)}{{\rm rk}(\E)}.
$$
A vector bundle $\E$ is said to be semistable (resp. stable) if 
$$\mu(\F)\le \mu(\E_L)\quad(\text{resp. } \mu(\F)< \mu(\E_L)) $$
for every subsheaf $\F$ of $\E_L$ as $L$ varies over all algebraic field extensions
of $\Bbbk$. An $\sl_n$-bundle is said to be 
semistable (resp. stable) if its associated vector bundle is so.

Given a vector bundle $\E$ set
$$
\mu = \sup \{ \mu(\E') | \E'\subseteq \E\}.
$$
One can show that there exists a unique subsheaf $\E_k$ of 
$\E$ such that $\mu= \mu(\E_k)$ and $\E_k$ is maximal with respect to 
inclusion amongst subsheaves of slope $\mu$, see \cite[Proposition 5.4.2]{lepotier:97}.
Such a sheaf is called a \emph{maximal destabilizing subsheaf}.
Induction yields a unique filtration
$$
0\subseteq \E_1 \subseteq \E_2\subseteq \ldots \subseteq \E_k\subseteq \E=\E_{k+1}
$$
such that
\begin{enumerate}
 \item The associated graded objects $\E_i/\E_{i-1} $ are semistable.
 \item The slopes $\mu(\E_i/\E_{i-1})>\mu(\E_{i+1}/\E_{i})$ are decreasing.
\end{enumerate}
This is the \emph{Harder-Narasimhan} filtration.

\begin{proposition}
\label{p:jordanholder}
Let $\E$ be a semistable bundle. There exists an increasing filtration, 
defined over a finite Galois extension $L/\Bbbk$, 
$$
\E_1\subset \E_2\subset \ldots\subset \E_n=\E_L
$$ 
such that $\E_i/\E_{i-1} =\gr_i(\E_\bullet)$ is stable. Moreover any two filtrations
$\E_i$ and $\E'_i$ have the same length
and there exists $\sigma\in S_n$ so that $\gr_i(\E_\bullet)\cong\gr_{\sigma(i)}(\E'_\bullet)$.
\end{proposition}

\begin{proof}
This is \cite[Proposition 5.3.7]{lepotier:97}.
\end{proof}

\begin{remark}
\label{r:ssunique}
It follows from the uniqueness of the Harder-Narasimhan filtration that
the field extension is not needed in the definition of semistable.
In other words, semistable may be defined in the following way,  a bundle $\E$
is semistable if $\mu(\F)\le \mu(\E)$ for all subbundles $\F$ of $\E$.
\end{remark}

This is not true for the notion of stable. For  example,
consider a curve $X/\Bbbk$ of genus at least one, and a quadratic
extension $L/\Bbbk$. We can arrange things so that there is
a point $p\in X(L)$ such that its Galois conjugate $p^\sigma$ is
different from itself. The rank two bundle 
$$
\E = \mathcal{O}(p)\oplus \mathcal{O}(p^\sigma)
$$
has a Galois action and descends to a bundle on $X$. However, its
Jordan-H\"older filtration exists only over the curve $X_L$.

Two bundles $\E$ and $\F$ are said to be \emph{S-equivalent} if the
$\oplus\gr_i(\E)$ and $\oplus\gr_i(\F)$ are isomorphic.

We summarise below some basic properties of stable and semistable bundles.

\begin{theorem}
\label{t:basic}
\begin{enumerate}

\item Let $\F$ be a stable bundle on $X$. Then $\coh{0}{X, \text{End}(\F)}$ 
is one dimensional. 
\item More generally, let $R$ be a ring and let $\F$ be a family of stable bundles
on $X_R$ parametrized by $R$, i.e. for every closed point $x$ of $\spec(R)$, 
the restriction of the family to $x$ is stable. Then $\coh{0}{X, \text{End}(\F)}=R$.
\item Being  stable and semistable are open conditions.
\item Fix a line bundle $\xi$ on $X$. There exists a  moduli space
$\su(X,n,\xi)$  of semistable bundles of rank $n$ and determinant $\xi$
on $X$. Its closed points correspond to S-equivalence classes of semistable bundles.
There is an open substack $\su^s(X,n,\xi)$ parameterising stable bundles.
\end{enumerate}
\end{theorem}

\begin{proof}
This is essentially carried out in Part I of \cite{lepotier:97} when 
$\Bbbk$ is algebraically closed. For our slightly more general setting
choose an algebraic closure $\Bbbk\subseteq\bar{\Bbbk}$. 

\noindent
(i) Suppose that we have an endomorphism  $\phi:\E\rightarrow\E$ of
a  stable bundle. We know that its base extension
$\phi_{\bar{\Bbbk}}$ is multiplication by a scalar $\lambda$. The
scalar $\lambda$ must come from $\Bbbk$ as $\phi$ is defined over
$\Bbbk$.

\noindent
(ii) There exists a natural inclusion
$$
\epsilon : R \hookrightarrow \coh{0}{X_R, \text{End}(\F)}
$$
that we wish to show is an isomorphism. By flat base change, we may assume $R=(R,m)$ is
local. Via Nakayama's Lemma we need to show that
$$
\bar{\epsilon} : R/\mathfrak{m} \hookrightarrow \coh{0}{X_R, \text{End}(\F)}\otimes_R R/\mathfrak{m}
$$
is surjective. But by (i), the composition 
$$
R/\mathfrak{m} \rightarrow \coh{0}{X, \text{End}(\F)}\otimes_R R/\mathfrak{m} \rightarrow
\coh{0}{X_{R/\mathfrak{m}}, \text{End}(\F_{R/\mathfrak{m}})}
$$
is surjective. The result follows from the base change theorem,
\cite[III Theorem 12.11]{hartshorne:77}.

\noindent
(iii) One may just adapt the proofs from \cite{lepotier:97} to our
situation or use the fact that for every scheme $S/\Bbbk$, the projection
$$
S_{\bar{\Bbbk}}\rightarrow S
$$
is an open morphism.

\noindent
(iv)  We wish to construct a moduli space for the stacks
$\bun^{s,\xi}_\sln$ and $\bun^{s,\xi}_\gln$. These are the open substacks of 
$\bun^{\xi}_\sln$ and $\bun^{\xi}_\gln$ parameterising stable bundles.
The result actually follows from a theorem of Keel and Mori~\cite{keel:97}, once we know
that $\bun^{s,\xi}_\sln$ is a Deligne-Mumford stack with finite inertia and
 $\bun^{s,\xi}_\sln$ and $\bun^{s,\xi}_\gln$ have the same moduli space. However we
will need an explicit description of the moduli space below. The family of
 stable bundles of given rank and determinant  is
a bounded family. This can be proved by passing to $\bar{\Bbbk}$ and applying
the result there. Hence there exists an integer $N$ such that 
$$
\coh{1}{X,\E(n)} = 0\quad\text{and}\quad \E(n)\quad\text{is generated by global sections}
$$
for every  stable bundle $\E$ of given rank and determinant and for every $n\ge N$.
Recall that we have assumed that our curve has a point $p$ over $\Bbbk$ so we define
$\E(n)=\E\otimes\struct{X}(np)$. Let 
$h=\dim \coh{0}{X,\E(N)}$ for a stable bundle of given rank and determinant.
Consider the quot scheme parameterising quotients
$$
\struct{X}(-N)^h\twoheadrightarrow \E
$$
with $\rk\E=n$ and $\deg\E=\deg\xi$. There is a locally closed subset $\Omega$
parameterising quotients (use (ii)) with $\E$ stable and $\det \E=\xi$.
Using part (i) we have
$$
\bun^{s,\xi}_\sln=[\Omega/\sl_h]\quad\text{and}\quad\bun^{s,\xi}_\gln=[\Omega/\gl_h].
$$
The center of $\gl_h$ acts trivially on $\Omega$, and is in fact the stabiliser of a point by (i).
It follows that we can identify the coarse moduli space with the quotient
$$
\bun^{s,\xi}_\gln\rightarrow \Omega/\gl_h=\su^s(X,n,\xi).
$$
\end{proof}

Note that a consequence of the above is that a family $\F$ of stable bundles
on $S\times X$ with $\det\F \otimes\text{pr}_X^*\xi^\vee$ being the pullback of a line
bundle on $S$, determines a morphism 
$$
\phi_\F : S \rightarrow  \su^s(X,n,\xi).
$$

We would like to record a kind of partial converse to the above:

\begin{proposition}

\noindent
(i) Given an $S$-point $\phi:S\rightarrow \su^s(X,n,\xi)$ there is an \'etale
cover $e:T\rightarrow S$ such that $\phi\circ e$ is equal to $\phi_\F$ for some
family on $T\times X$.

\noindent
(ii) Suppose that we have two families $\F_1$ and $\F_2$ on $S\times X$ that
determine the same $S$-point of $\su^s(X,n,\xi)$. Then there is an \'etale cover
$T\rightarrow S$ such that pullbacks $\F_{1,T}$ and $\F_{2,T}$ are isomorphic.
\end{proposition}

\begin{proof}
Using the notation of the proof of the preceding proposition we note that
$\Omega\rightarrow \Omega/\text{PGL}_h=\su^s(X,n,\xi)$ is a 
$\text{PGL}_h$-principal bundle. To see this, we note that
the question is local in the \'etale topology, so we may pass to an algebraically closed 
field and use the known result there. To finish off the proof, recall that for any 
$\text{PGL}_h$-principal bundle $P\ri B$, an $S$-point $S\ri B$ lifts to $P$ upon 
passing to an \'etale cover of $S$.
\end{proof}

\section{Basic properties of our moduli stacks} \label{s:stack}

Recall that we denote by $\bun^{s,\xi}_\sln$ and $\bun^{s,\xi}_\gln$ the open
substacks of $\bun^{\xi}_\sln$ and $\bun^{\xi}_\gln$
parameterising stable bundles.

There is an obvious morphism of stacks
$$
\bun^{s,\xi}_\sln\rightarrow\bun^{s,\xi}_\gln
$$
that  will give a $u$-morphism of gerbes below.

\begin{lemma}\label{l:glgerbe}
The natural map $\bun^{s,\xi}_\gln \ri \su(X,n,\xi)$
makes the stack into a gerbe banded by $\gm$ over the moduli space.
\end{lemma}

\begin{proof}
 The previous proposition is saying that it is a gerbe. The band is
computed in \ref{t:basic}.
\end{proof}

Similarly we have : 

\begin{lemma}
The natural map $\bun^{s,\xi}_\sln \ri \su(X,n,\xi)$
makes the stack into a gerbe banded by $\bbmu_n$ over the moduli space.
\end{lemma}

We need to show that $\bun^{s,\xi}_\sln$ is a Deligne-Mumford
stack with finite inertia.

\begin{proposition}
The stack $\bun^{s,\xi}_\sln$ is a Deligne-Mumford stack.
\end{proposition}

\begin{proof}
The stack is of finite type as the collection of stable
bundles forms a bounded family, see \cite[Chapter 7]{lepotier:97}. We need to show that the diagonal
is formally unramified. Consider an extension of Artinian local
$k$-algebras
$$
0\rightarrow I \rightarrow A' \rightarrow A \rightarrow 0.
$$
An $A'$-point of $\bun^{s,\xi}_\sln \times \bun^{s,\xi}_\sln$
amounts to two families $(\F_1, \phi_1)$ and $(\F_2,\phi_2)$ of stable
bundles with identifications of their top exterior powers with $\xi$ 
parametrised by
$A'$. Completing this to a diagram of the form 
$$
\xymatrix{
 \spec(A) \ar[r] \ar[d] &  \spec(A') \ar[d] \\
 \bun^{s,\xi}_\sln \ar[r] & \bun^{s,\xi}_\sln \times \bun^{s,\xi}_\sln
}
$$
amounts to an isomorphism 
$\alpha:\F_1|_A \cong \F_2|_A$ compatible with the identifications of the
top exterior powers. We need to show that any extension of the isomorphism
$\alpha$ to $A'$ is unique. In view of \ref{t:basic} this follows from the following claim

\begin{claim}
 Let $(B,\mathfrak{m})$ be a local $\Bbbk$-algebra. Suppose that
$y_i\in B$ and 
$y_1^n=y_2^n = 1$. Further assume that $y_i$ have the same images under
the projection 
$$
q:B\rightarrow B/\mathfrak{m}.
$$
Then $y_1=y_2$.
\end{claim}

\textit{Proof of claim.} 
We may write $y_2 = y_1 + x$
where $x\in \mathfrak{m}$. As we are in characteristic 0, we have
$$
1 = (y_2)^n + x(\text{another unit in } B).
$$
Since $y_2^n=1$ we must have $x=0$.
\end{proof}

In order to make use of the work in \cite{brosnan:07}
we need to see that the morphism
$$
{\mathcal I}(\bun^{s,\xi}_\sln)\rightarrow  \bun^{s,\xi}_\sln,
$$
where ${\mathcal I}(\sX)$ means inertia stack, is a finite
morphism. 

\begin{proposition}
The stack $\bun^{s,\xi}_\sln$ has finite inertia.
\end{proposition}

\begin{proof}
Using \ref{t:basic},
one identifies the inertia stack with
$\bun^{s,\xi}_\sln\times_\Bbbk \bbmu_n$. Hence the projection
$$ \bun^{s,\xi}_\sln\times_\Bbbk \bbmu_n \rightarrow \bun^{s,\xi}_\sln $$
is a finite morphism. 
\end{proof}

\section{Twisted sheaves and the Brauer group}

This section collects some general results about
the Brauer group and twisted sheaves. Let $X/\Bbbk$ be a scheme.
A gerbe $\sG\rightarrow X$ banded by $\bbmu_n $ gives a class $[\sG]$ in 
$\coh{2}{X_\et,\bbmu_n}$ and hence a torsion class in
$\coh{2}{X_\et,\gm}$. Recall that the \emph{period} of $\sG$ is defined to
be the order of this class. If $X=\spec(K)$ for a field $K$
we define the \emph{index} of $[\sG]$ to be the greatest common divisor 
of the degrees of splitting fields of $[\sG]$. 

The following is well-known.

\begin{proposition} \label{p:periodindex}
When $X=\spec(K)$ in the above situation the period divides
the index.
\end{proposition}

\begin{proof}
This is well known, for example see \cite[Proposition 4.16]{farb:93}.
\end{proof}

A useful tool for understanding the difference
between the period and the index is the notion of a
twisted sheaf. A \emph{twisted sheaf} on a $\gm$-gerbe
$\sG\rightarrow X$ is a coherent sheaf $\F$ on $\sG$
such that inertial action of $\gm$ on $\F$ coincides with
natural module action of $\gm$ on $\F$. We spell out the meaning of
this statement
in the next paragraph.

 Suppose that we have a $T$-point $T\rightarrow X$ and an object
$a$ of $\sG$ above this point. Part of the data of the coherent
sheaf $\F$ is a sheaf $\F_a$ on $T$. These sheaves are required to
satisfy compatibility conditions on pullbacks for morphisms in the category
$\sG$. In particular, every object $a$ of the gerbe $\sG$ has an action
of $\gm$ and hence there is an action of $\gm$ on $\F$. The above 
definition says that action of $\gm$ on $\F$ should be the same as the $\gm$-action coming
from  the fact that $\F$ is an $\struct{\sG}$-module.

\begin{example} \label{e:twisted}
We have a $\bbmu_n$-gerbe 
$$
\bun^{s,\xi}_{{\rm SL}_n} = [\Omega/\sl_h]\rightarrow [\Omega/{\rm PGL}_h] = \su(X,n,\xi).
$$
It gives rise to a $\gm$-gerbe
$$
\bun^{s,\xi}_{{\rm GL}_n} = [\Omega/{\rm GL}_h]\rightarrow [\Omega/{\rm PGL}_h] = \su(X,n,\xi),
$$
where $\bun^{s,\xi}_{{\rm SL}_n}$ is the moduli stack of bundles with  determinant $\xi$
but the isomorphisms do not induce the identity on the determinant.
The universal bundle on $\bun^{s,\xi}_{{\rm GL}_n}\times X$ is a twisted sheaf since
the only automorphisms of a stable bundle are given by multiplication by a scalar.
\end{example}

We will need the following :

\begin{proposition} \label{p:twisted}
Let $\sG\rightarrow\spec(K)$ be a $\gm$-gerbe over a field.
Then the index of $\sG$ divides $m$ if and only if there is 
a locally free rank $m$ twisted sheaf on $\sG$.
\end{proposition}

\begin{proof}
See \cite[Proposition 3.1.2.1]{lieblich:08}.
\end{proof}

\begin{corollary}\label{c:unitwist}
The index of the gerbe $\bun^{s,\xi}_{{\rm GL}_n}\rightarrow\su(X,n)$
over the generic point of $\su(X,n,\xi)$ divides $n$.
\end{corollary}

\section{The Brauer group of $\su(X,n,\xi)$}

In this section we recall the results of 
\cite{balaji:07} and we present some minor modifications
of these results for our own context.

There is a natural Severi-Brauer variety over
$\su(X,n,\xi)\times X$. To construct it, using the notation of \S \ref{s:stack}, notice that the
${\rm PGL}_h$ action on $\Omega$ lifts to the projectivisation
of the universal bundle on the quot scheme. Let 
$\Pr$ be the quotient Severi-Brauer variety.
Each closed point $x\in X$ gives an inclusion
$$
\su(X,n,\xi)\hookrightarrow \su(X,n)\times X.
$$
Denote by $\Pr_x$ the pullback of $\Pr$ via this inclusion.

\begin{proposition}\label{p:balajietal}
When working over $\Bbbk = {\mathbb C}$ we have :
\begin{enumerate}
\item The Brauer group ${\rm Br}(\su(X,n,\xi))$ is cyclic of order 
$\gcd(n,\deg(\xi))$.

\item The Brauer group is generated by the class of the gerbe
$$\bun^{s,\xi}_\gln\rightarrow \su(X,n,\xi).$$

\item The class of the Brauer-Severi variety 
$\Pr_x$ in $\coh{2}{\su(X,n,\xi),\gm}$ coincides with the
class of the gerbe 
$$\bun^{s,\xi}_\gln\ri \su(X,n,\xi).$$ 
This class does not  depend on $x$.
\end{enumerate}
\end{proposition}

\begin{proof}
This is \cite[Theorem 1.8]{balaji:07} and the discussion
immediately before it.
\end{proof}

In our setting we are not working over the complex 
numbers but we do not need the full power of the
result above.  We can prove the following which is sufficient
for our question on essential dimension.

\begin{proposition}
\label{p:periodk}
 The period of the gerbe 
$$\bun^{s,\xi}_\gln\rightarrow \su(X,n,\xi)$$
is $\gcd(n,\deg(\xi))$.
\end{proposition}

To prove this we need to recall some constructions from  \cite{drezet:89}.
Recall that the moduli spaces $\su(X,\xi,n)^s$ and $\su(X,\xi,n)^{ss}$
can be constructed as geometric invariant theory quotients
\[
 \su(X,\xi,n)^s = \Omega/\glq \quad 
\su(X,\xi,n)^{ss} = \Omega^{ss}\!/\!/\glq \quad 
\]
where $\Omega^*$ is an appropriate open subset of the quot scheme 
as in section 4. We write $\Omega^*$ to mean one of $\Omega$ or
$\Omega^{ss}$.  Let $L$ be a $\glq$-line bundle on $\Omega^*$.  
There is an integer $e(L)$ such that the center of 
$\glq$ acts on $L$ with weight $e(L)$.

\begin{proposition}
 Let $k$ be an integer. There exists a $\glq$ line bundle on 
$\Omega$ with $e(L)=k$ if and only if $k$ is divisible 
by $\gcd(\deg(\xi),n)$ 
\end{proposition}

\begin{proof}
This is precisely proposition 5.1 of \cite{drezet:89}. There it was proved 
over the complex numbers but the proof goes through in our 
case. We briefly outline it here for the convenience of the reader.

First consider the reverse implication. We have  a universal bundle
$U$ on $\Omega\times X$. The result follows by considering the weight of
central torus actions on the line bundles
\[
 \det(i^*(U))\quad\text{and}\quad \det(\pi_*(U\otimes{\mathcal O}_X(m)).
\]
Here $i:\Omega\hookrightarrow \Omega\times X$ is the inclusion
at some point of $X$ and $\pi: \Omega\times X\rightarrow \Omega$ the
projection.

For the other direction one can simply observe that $e(L)$ doesn't change 
under a base extension
$$
\spec(K)\ri \spec(\Bbbk).
$$
So one may base change to an algebraically closed field and use a 
Lefschetz principle.
\end{proof}

\begin{proof} (\emph{of \ref{p:periodk}})
 With the above lemma the proof can be now copied from
\cite{balaji:07}.
\end{proof}

\section{The period index problem for our gerbe}

Let $K$ be the function field of $\su(X,n,\xi)$. We have a 
gerbe over $K$ defined by the 2-Cartesian square
$$
\xymatrix{
\sG \ar[r] \ar[d] & \bun^{s,\xi}_{{\rm GL}_n} \ar[d] \\
 \spec(K) \ar[r] & \su^s(X,n,\xi).
}
$$ 
Set $d=\gcd(n,\deg(\xi))$. We know that the period of
$\bun^{s,\xi}_{{\rm GL}_n}$ is $d$. 
Let us remark that the period
of $\sG$ is also $d$. This follows from the following two
facts.

\begin{proposition}
 Let $X$ be a regular scheme with function field
$K$. The pullback map 
$$
Br(X)\rightarrow Br(K)
$$
is injective.
\end{proposition}

\begin{proof}
 See \cite[IV Corollary 2.6]{milne}.
\end{proof}

\begin{proposition}
 The moduli space $\su^s(X,n,\xi)$ is a smooth 
algebraic variety.
\end{proposition}

\begin{proof}
 By \cite[Theorem 1.1]{git}, geometric invariant theory quotients are
uniform. So 
$$
\su^s(X,n,\xi)_{\bar{\Bbbk}} = \su^s(X_{\bar{\Bbbk}},n,\xi).
$$
Hence one may base change to an algebraically closed field and apply the result 
there, see \cite[Chapter 8]{lepotier:97}.
\end{proof}

By \ref{p:periodindex}, \ref{p:balajietal} and the above discussion
we know that $d$ divides the index of $\sG$.

In fact we have :

\begin{proposition}\label{p:index}
We have $d={\rm ind}(\sG)$ so that period equals the index for
this gerbe.
\end{proposition}

\begin{proof}
It suffices to show that ${\rm ind}(\sG)$
divides $n$ and $\deg(\xi)$. It follows from
\ref{e:twisted} and \ref{p:twisted} that the index divides $n$.

Recall that $X$ has a point.
Taking $\cL=\struct{X}(d)$ for $d$ large we may assume that
$$ R^1\pi_*(\F^{{\rm univ}}\otimes \pi^*_X\cL)=0,$$
where $\F^{{\rm univ}}$ is the universal bundle on 
$$\bun^{s,\xi}_{{\rm GL}_n} \times X$$
and $\pi$ the projection onto $\bun^{s,\xi}_\gln$.
As we are over the stable locus, the bundle 
$\pi_*(\F^{\rm univ})$ is a twisted sheaf of rank
$$
\chi = \deg(\xi) + n(1-g).
$$
Applying \ref{p:twisted} again the result follows.
\end{proof}

\section{The stable locus}

\begin{proposition}
Let $\alpha$ be  the class of $\bun^{s,\xi}_\sln$ inside
$\coh{2}{\su^s(X,\xi,n),\bbmu_n}$. The image of $\alpha$
under the natural map 
$$
\coh{2}{\su^s(X,\xi,n),\bbmu_n}\ri 
\coh{2}{\su^s(X,\xi,n),\gm}
$$ 
is the class of $\bun^{s,\xi}_\gln$.
\end{proposition}

\begin{proof}
We have a natural inclusion 
$u : \bbmu_n \hookrightarrow \gm$
and a diagram
$$
\xymatrix{
 \bun^{s,\xi}_\sln \ar[rr]^\phi \ar[dr] & & \bun^{s,\xi}_\gln \ar[dl] \\
  	& \su(X,n,\xi). &
}
$$
The map $\phi$ is a $u$-morphism in the sense of 
\cite[Ch. IV 2.1.5]{giraud:71}. The result now follows
from \cite[Ch. IV 3.1.5]{giraud:71}.
\end{proof}

\begin{theorem}\label{t:bound}
Suppose that ${\rm char}(\Bbbk)=0$.
We have a bound
$$
\ed(\bun^{s,\xi}_{\sln}) \le (n^2-1)(g-1) + d,
$$
where $d=\gcd(n,\deg(\xi))$.
This inequality is an equality when 
$d=p^r$ is a prime power.
\end{theorem}

\begin{proof}
Let $K$ be the function field of $\su(X,n)$ 
and $\sG\rightarrow\spec(K)$ the generic gerbe 
defined by the Cartesian diagram
$$
\xymatrix{
\sG \ar[r] \ar[d] & \bun^{s,\xi}_{\sln} \ar[d] \\
 \spec(K) \ar[r] & \su(X,n,\xi).
}
$$ 
By \ref{t:main} we have
$$
\ed(\bun^{s,\xi}_{\sln}) = \dim \su(X,n,\xi) + \ed(\sG/K),
$$
and $\dim \su(X,n,\xi) = (n^2-1)(g-1)$, see \cite[Theorem 8.3.2]{lepotier:97}.

It remains to understand the essential dimension of the generic
gerbe. By \ref{t:main2} we have
$$
\ed(\sG/K) = e(SB) + 1,
$$
where $SB$ is a Severi-Brauer variety of dimension ${\rm ind}(\sG) - 1$. The index of the generic
gerbe is computed in 
 \ref{p:index}. 
Recall $e(X)$ is the minimum element of the set
$$
\{ \dim \overline{{\rm Im}(\phi)}\mid \phi\text{ a rational endomorphism of }X\}.
$$ 
It follows that $e(SB)\le \dim_K(SB) \le n-1$.

For the equality  one applies the corollary to \ref{t:main2}
which states that
$$
\ed(\sG/K) = \text{ index of }\sG/K = n,
$$
when $n$ is a prime power.

\end{proof}

\section{The Galois theory of stable bundles}

Fix a Galois extension $L/K$ with Galois group $G$.
Let $\E$ be a semistable bundle on $X_K$ with slope 
$\mu$. We shall abuse notation and write $\E$ for the pullback
to $X_L$. Note that there are canonical identifications
$h^*\E\cong \E$ for every $h\in G$. 

Let $V$ be a stable bundle on $X_L$ of slope $\mu$ 
and suppose that ${\rm Hom}(\E,V)$ is non-zero.
Let $q=\dim({\rm Hom}(\E,V))$.
Choose an ordered basis $\phi_1,\phi_2,\ldots \phi_q$
for ${\rm Hom}(\E,V)$. We will need the fact that the
induced map 
$$
\E\rightarrow V^q
$$
is surjective. This follows from

\begin{proposition}
\label{p:surjects}
Let $\E$  be a semistable bundle
and let $V$ be a stable bundle of 
the same slope. Suppose
$$\psi_1,\ldots,\psi_k \in {\rm Hom}(\E,V)$$ are linearly independent.  
 Then the induced map
$$
\E\rightarrow V^k
$$
 is surjective.
\end{proposition}

\begin{proof}
One inducts on $k$. In the case $k=1$, since $\E$ is semistable, $V$
is stable and both have slope $\mu$,
we see that 
$$\mu=\mu(\E)\le \mu(\psi_1(\E))\le\mu(V)=\mu$$
and so $\mu(\psi_1(E))=\mu(V)=\mu$ which implies, from the 
stability of $V$ that $\psi_1(\E)=V$.  So $\psi_1$ is surjective.

In general, let $\K$ be the kernel of 
$\psi_k$. By the previous argument, $\psi_k$ is surjective.
As we have an exact sequence
$$
0 \rightarrow \text{Hom}(V,V) \rightarrow
\text{Hom}(\E,V) \rightarrow \text{Hom}(\K,V),
$$
$\psi_1,\psi_2,\ldots, \psi_{k-1}$ restrict
to linearly independent homomorphisms from $\K$ to $V$.  Then
one applies the induction hypothesis to $\K$.
\end{proof}

We write $\Phi : \E\rightarrow  V^q$ for the surjection induced by the basis
$\phi_1,\phi_2,\ldots \phi_q$ and $\K$ for its kernel.
 
We have for each $h\in G$ a composition of
surjective maps 
$$ \E\isoarrow h^*\E\stackrel{h^*\Phi}{\rightarrow}(h^*V)^q.$$
We abuse notation and write $h^*\Phi$ for the composition of these
two maps.
Note that $h^*(\K)=\Ker(h^*\Phi)$
so that for each $g\in G$,
we have a short exact sequence
$$0\to h^*(\K)\to \E\stackrel{h^*\Phi}{\to} h^*(V)\to 0$$

\begin{proposition}
Suppose that we are given  different basis 
$$\psi_1,\psi_2,\ldots\psi_q\in {\rm Hom}(\E,V).$$ Then there
exists a unique automorphism
$$ \alpha^{\Phi,\Psi}_h=\alpha_h: (h^*V)^q\isoarrow (h^*V)^q $$
such that the following diagram commutes
$$
\xymatrix{
\E \ar[r]^<<<<{h^*\Phi} \ar@{=}[d] & (h^*V)^q \ar[r] \ar[d]_{\wr}^{\alpha_h} & 0\\
\E \ar[r]^<<<<{h^*\Psi} & (h^*V)^q \ar[r] & 0.
}
$$
These isomorphisms are functorial with respect to $h$ that is
$h^*\alpha_{h'} = \alpha_{hh'}.$
\end{proposition}

\begin{proof}
The uniqueness is clear. We first construct $\alpha_{{\rm id}}$. 
In this case there is an $\alpha\in {\rm GL}({\rm Hom}(\E,V))$ that
sends the basis $\Phi=\{\phi_j\}$ to the basis $\Psi=\{\psi_j\}$. Then we take 
$\alpha_{{\rm id}}$ to be the induced automorphism of the 
polystable bundle $V^q$. One defines $\alpha_h$ to be
$h^*\alpha_{{\rm id}}$. 
\end{proof}

 Let $S$ be the stabilizer of the $G$-action on $V$, that is
$$
S =\{g\in G| g^*V\cong V\}.
$$
Also
 let 
$\id=h_1,h_2,\ldots,h_l$ be coset representatives for $G/S$.

\begin{proposition}\label{p:rigid}
Let $h\in G$ and suppose that $hS=h_iS$ for some $i=1,\dots,l$.
Choose an isomorphism $\beta:h^*V\stackrel{\cong}{\to} h_i^*V$. Then there exists a unique
isomorphism $(h^*(V))^q\cong (h_i^*(V))^q$ such that the following diagram
commutes
\[
\xymatrix{
h^*\E \ar[r] \ar@{=}[d] & (h^*V)^q \ar[r] \ar[d] & 0\\
h_i^*\E \ar[r] & (h_i^*V)^q \ar[r] & 0.
}
\]
\end{proposition}

\begin{proof}
The isomorphism is the composite $\beta^q\circ\alpha_h$ where
$\alpha_h$ is the isomorphism from the last proposition.
\end{proof}

Let us recall a definition.
\begin{definition}
 \label{d:action}
Suppose a finite group $H$ acts on a scheme $Y$. Let 
$\F$ be a sheaf on $Y$. We say $\F$ (really $(\F,\alpha_g)$) is a \emph{$H$-sheaf}
if there are isomorphisms
$$
\alpha_g:g^*\F\isoarrow \F
$$
for each $g\in G$ subject to the conditions

\noindent
(1) $\alpha_1=\text{identity}.$

\noindent
(2) For every $g,h\in G$ the following diagram commutes
$$
\xymatrix{
g^*h^*\F \ar[r]^{g^*\alpha_h} \ar[dr]_{\alpha_{hg}} & g^*\F \ar[d]^{\alpha_g} \\
     & \F.
}$$
\end{definition}

\begin{corollary} 
\label{c:action} The coherent sheaf $V$ is an $S$-sheaf, that is, there
is an action of the group $S$ on $V$ compatible with the action
of $S$ on $X_L$.
\end{corollary}

\begin{proof}
This is because $\E$ is an $S$-sheaf and the uniqueness part of \ref{p:rigid}.
\end{proof}

\begin{proposition}
\label{p:jh}
Let $V$ be a stable bundle with the same slope
as the semistable bundle $\E$. Set $q = \dim {\rm Hom}(\E,V)$.
If the associated graded bundles of the Jordan-H\"older filtration
of $\E$ are $\G_1,\ldots, \G_\alpha$ then at least $q$ of the $\G_i$ are isomorphic
to $V$.
\end{proposition}

\begin{proof} Let $\phi_1, \phi_2,\ldots ,\phi_q$ be a basis for 
${\rm Hom}(\E,V)$.
Using the discussion at the start of this section we obtain an exact sequence
$$
0\rightarrow \K\ri\E\ri V^q \ri 0.
$$
Then we can obtain a Jordan-H\"older filtration of $\E$ by extending
such a filtration of $\K$. Precisely, if $\E_1,\ldots,\E_k$ form a Jordan H\"older
filtration of $\K$ then for each $j=1,\dots,q$,
define $\E_{k + j}$ to be the kernel of the surjective map
$$
\E \ri V^{q-j}
$$
given by $\phi_1,\ldots ,\phi_{q-j}$. The result now follows.
\end{proof}

\begin{proposition}
\label{p:keybound}
We have
$$
|G/S| \le \frac{ \rk(\E)  }{ q . \rk(V) },
$$
where $q=\dim {\rm Hom}(\E,V)$ and $S$ is the stabilizer subgroup of $V$
in $G$.
\end{proposition}

\begin{proof}
Let $\{\G_i\}$ be a set of associated graded bundles of the  Jordan-H\"older filtration of $\E$.
By Proposition \ref{p:jh} applied to the semistable bundle $\E$ and the 
stable bundle $h_i^*V$ for each coset representative $h_i$ of $S=\Stab_G(V)$ in $G$,
we see that we have $q$ of the $\G_i$ isomorphic to $h_i^*V$ for 
each $i=1,\dots,[G:S]$. Then $\rk(\E)\ge q.[G:S]\rk(V)$.
The result follows.
\end{proof}

For the following corollary we will make use of Galois descent.
An introduction to this subject can be found in \cite[pg. 60]{algv}
and \cite[Ch 2]{knus:74}. As stated the theorems in these two
references are not quite general enough for our purposes.
A very general version of this theorem  is written down in
\cite[pg. 19]{milne}.  The relationship of this last theorem to Galois
descent is established by realizing that for a Galois cover
$S\ri T$ with group $H$ we have $S\times_T S\cong S\times H$.

\begin{corollary}
\label{c:keybound}
There is a field extension $L'/K$ of degree at most
 $$
 \frac{ \rk(\E)  }{ q . \rk(V) },
$$
over which $V$ is defined. Furthermore there is a 
surjection 
$$
\E\rightarrow V^q\rightarrow 0
$$
defined over $L'$
\end{corollary}

\begin{proof}
 One takes $L'=L^S$ and applies \ref{c:action} above to see
that $V$ descends to $L'$. Note that
$$
\hom{X_{L'}}{\E_{L'}, V_{L'}}\otimes L = \hom{X_{L}}{\E_L, V_L}
$$
so one takes a new basis defined over $L'$ and applies
\ref{p:surjects}. Note that we are not asserting that
the original surjection descends to $L'$.
\end{proof}

\begin{remark}\label{r:gen}
In order to obtain a bound on the essential dimension we replace
$L$ with the Galois closure of $L'/K$. 
 So $G$ is some
subgroup of the symmetric group $S_p$ with 
$p=\dim L'/K$ and hence by the corollary is a subgroup
of $S_{\rk(\E)}$. 
\end{remark}

Let us record the following result.

\begin{proposition}
\label{p:ssbound}
 Let $\F$ be a semistable vector bundle of rank $n$ and  degree $d$
over our curve $X$ of genus $g$. Then
$h^0(\F)\le \max(d/n+1,0)n$. Furthermore, when $\F$ has non-negative
slope we have  $h^0(\F)\le n+d$ and 
$h^1(\F)\le ng$.
\end{proposition}

\begin{proof}
The first part is is Lemma 7.1.2
\cite{lepotier:97}. The second statement follows from the
first via Riemann-Roch.
\end{proof}

\begin{corollary}
\label{c:extbound}
Let $\E$ be a non-stable vector bundle of rank $n$ over $X$.
Let $\E'$ be a maximal destabilizing proper subbundle with $\mu(\E') > \mu(\E)$
and rank $n'<n$.
Then 
$$\dim(\Ext^1(\E/\E',\E'))\le n'(n-n')g$$
\end{corollary}

\begin{proof}

Suppose that the Harder-Narasimhan filtration of $\E$ is
$$0\subseteq \E_1\subseteq \E_2\subseteq \dots \subseteq \E_k=\E$$
so that $\E_{k-1}=\E'$. The Harder-Narasimhan filtration of 
$$
\E' \otimes (\E/\E')^\vee
$$ 
is then 
$$0\subseteq \E_1\otimes (\E/\E_{k-1})^\vee\subseteq
 \E_2\otimes (\E/\E_{k-1})^\vee\subseteq \dots \subseteq 
\E_{k-1}\otimes (\E/\E_{k-1})^\vee=\E'\otimes (\E/\E_{k-1})^\vee.$$

Notice that
$$
\mu((\E/\E_{k-1})^\vee \otimes \E_i/\E_{i-1}) =
\mu(\E_i/\E_{i-1}) - \mu(\E/\E_{k-1})
$$
which is positive and the bundle $(\E/\E_{k-1})^\vee \otimes \E_i/\E_{i-1}$ is
semistable so that the proposition applies to it. A long exact sequence and simple induction 
completes the proof.
\end{proof}

\section{From stable to semistable}

We begin with a couple of simple observations.

\begin{lemma}
 Let $L/L_1$ be a field extension and consider the 
morphism
$$
f:X_L \ri X_{L_1}.
$$
Let $\F,\G$ be coherent sheaves on $X_{L_1}$ and suppose we have two 
morphisms
$$
\alpha_i: \F\ri \G\qquad i=1,2.
$$ 
If $f^*\alpha_1 = f^*\alpha_2$
then $\alpha_1=\alpha_2$.
\end{lemma}

\begin{proof}
 Note that the morphism $f$ is flat. Hence
$$
\hom{X_{L_1}}{\F,\G}\otimes L = \hom{X_L}{f^*\F,f^*\G}.
$$
\end{proof}

\begin{proposition}
\label{p:gen}
 In the situation of the above lemma suppose that a finite
group $G$ acts on both $L$ and $L_1$. Suppose further that 
$f^*\F$ is a $G$-sheaf and we have associated isomorphisms
$$
\alpha_g:g^*\pi^*\F\isoarrow \F.
$$
Consider $\{g_1,g_2,\ldots, g_k\}$ a generating set for 
$G$. Suppose that there exist isomorphisms 
$$
\beta_{g_i} : g_i^*\F\isoarrow \F
$$
with $f^*(\beta_{g_i}) = \alpha_{g_i}$
then there is a $G$-action on $\F$ that pulls back to the
$G$-action on $f^*\F$.
\end{proposition}

\begin{proof}
 For every  $g\in G$, we fix an expression
$$
g= g_{i_1}g_{i_2}\ldots g_{i_k}.
$$
We define
$$
\beta_g : g^*\F\isoarrow \F
$$
as the composition of the following list of morphisms
$$
\begin{array}{rcl}
g_{i_k}^*\ldots g_{i_3}^*g_{i_2}^*(\beta_{g_1}) &:& g_{i_k}^*\ldots g_{i_2}^*g_{i_1}^*\F {\ri}  g_{i_k}^*\ldots g_{i_3}^*g_{i_2}^*\F \\
\vdots\quad & &\quad \vdots \\
g_k^*\beta_{g_{k-1}} & : & g_k^*g_{k-1}^*\F \isoarrow g_{k-1}^*\F \\
\beta_{g_k} & : & g_k^*\F \isoarrow \F 
\end{array}
$$
We have $f^*(\beta_g) = \alpha_g.$ To check that this is an action we need to
see that the conditions of \ref{d:action} hold. As they hold for $\alpha_g$ and
 $\beta_g$ pullbacks to $\alpha_g$, they hold for $\beta_g$ by the lemma.
\end{proof}

To obtain the bound we will need to make use of the
following construction :

\begin{proposition}
Fix  a projective scheme $Y$ and another scheme $Q$ over $\Bbbk$.
 Let $\F$ and $\G$ be families of coherent sheaves on 
$Y\times Q$. Consider the functor
$$
\isomF(\F,\G):\textbf{Schemes}/Q\rightarrow \textbf{Sets}
$$
whose value on $f:P\rightarrow Q$ is the set of
isomorphisms :
$$
\isomF(\F,\G)(f:P\ri Q) =
\{\alpha:(f\times 1)^*(\F)\stackrel{\sim}{\ri}
(f\times 1)^*(\G).
$$
This functor is representable by a scheme
$\isom(\F,\G) \ri Q$.
\end{proposition}

\begin{proof}
 See \cite[pg. 29, proof of theorem 4.6.2.1]{laumon:00}
\end{proof}

\begin{remark}
 \label{r:zariski} Suppose $Q=\spec(K)$, where $K$ is a field.
Consider an isomorphism $\alpha:\F\isoarrow\G$.
The Zariski tangent space to $\isom(\F,\G)$ at $\alpha$ can be identified with
$\hom{}{\F,\G}$. The reason is that a tangent vector is just a
morphism 
$$
\spec(K[\epsilon])\ri \isom(\F,\G)
$$
lifting the $K$-point induced by $\alpha$. (Here $K[\epsilon]=K[t]/t^2$.) This is just
a matrix of maps
$$
\left(
\begin{array}{cc}
\alpha & d \\
0      & \alpha
\end{array} \right)
:\F\oplus \epsilon\F \isoarrow \G\oplus \epsilon\G.
$$
The only unknown parameter is $d$ which is just a homomorphism
$d:\F\ri\G$.
\end{remark}

\begin{construction} \label{c:field}
 Let us recall the set-up of the previous section. Consider a
semistable bundle $\E$ of rank $n$ on $X_K$. By passing to 
a Galois extension $L/K$ with Galois group $G$ we can find an exact 
sequence, defined over $L$,
\begin{equation}
 0\ri \K \ri \E\ri V^q\ri 0. \tag{E}
\end{equation}
 Suppose
that $K$ and $V$ descend to a subfield $\widetilde{N}$ of $L$. We may replace $\widetilde{N}$ by $N$ its
Galois closure in $L$. Set
$$
W = \text{Ext}^1_{X_{N}}(V^q,\K).
$$
There is a universal extension 
$$
 0\ri \K \ri\E^\univ\ri V^q\ri 0
$$
on $W\times X$. (The universal extension amounts to constructing a universal cohomology class. 
Let $R$ be an $N$-algebra. By \cite[III, 12.11]{hartshorne:77},
$$
\text{Ext}^1_{X_{R}}(V^q,\K) = W \otimes R,
$$
and hence a cohomology class gives a homomorphism
$
R\otimes \text{Sym}^\bullet W^\vee \ri R.
$
The class is constructed by thinking of 
 $W$ as $\spec(\text{Sym}^\bullet W^\vee)$. )
Using \cite{ramanan:71}, the universal extension descends
to an extension on $\Pr(W)\times X$ of the form
$$
 0\ri \K\otimes\struct{}(1) \ri\E^\univ\ri V^q\ri 0.
$$
There is a morphism $\phi_E:\spec(L)\ri\Pr(W)$ with
$\phi_E^*(\E^\univ)=\E$. The generic point of the image of 
$\phi_E$ is of the form $\spec(\widetilde{M})$ for some subfield $\widetilde{M}$of $L$.
We let $M$ be its Galois closure in $L$.
We have a diagram of fields
$$
\xymatrix{
 L & M \ar@{_{(}->}[l]\\
 K \ar@{^{(}->}[u] & M^G \ar@{_{(}->}[l] \ar@{^{(}->}[u]
}
$$
with $\E$ defined over $M$. Finally we need to construct an extension $L_1$ of $M$
so that the Galois action on $\E$ descends to $L_1$. By \ref{p:gen} we only need to
make the generators descend. By \ref{r:gen}  we can assume that $G$ is a subgroup of $S_n$ and hence by
\cite[Theorem 1.13]{permutation} it can be generated by $n-1$ group elements.
Choose a generating set
$\{g_1,g_2,\ldots,g_{n-1}\}$ for $G$. Consider the scheme
$$
\isom(\E:g_1,\ldots, g_{n-1}) \stackrel{\text{def}}{=}
\isom(\E,g_1^*\E)\times_M\ldots\times_M\isom(\E,g_{n-1}^*\E).
$$
The isomorphisms $g_i^*\E\isoarrow \E$ defining the Galois action
of $\E$ are defined over the function field of some point of 
$\isom(\E:g_1,\ldots, g_{n-1})$ by \ref{p:gen}.
\end{construction}

\begin{proposition} \label{p:definition}
 In the above situation, suppose that $V$ and $\K$ descend to a field
$N$ with $\trdeg N/\Bbbk=p$. 
Then there is a subfield $L_1$ of $L$, stable under $G$, to which
$\E$ with its Galois action descends. Furthermore, we have
$$
\trdeg L_1/\Bbbk \le q.\rk(V)(n-q.\rk(V))g-1 +n^3-n^2 +p.
$$ 
\end{proposition}

\begin{proof}
 In \ref{c:field}, we constructed $L_1$. So we just need to count
transcendence degrees. The transcendence degree of $M$ is bounded
by $\dim W-1 +p$. Noticing that
$V^q$ and $\K$ are semistable of the same slope we have by \ref{p:ssbound}, 
$$\trdeg M/\Bbbk\le q.\rk(V)(n-q.\rk(V))g-1 +p.$$
By construction~\ref{c:field} and by remark~\ref{r:zariski},
we have
\begin{eqnarray*}
\trdeg L_1/M &\le &\dim(\isom(\E:g_1,\ldots,g_{n-1})) \\
             &= &\sum_{i=1}^{n-1}\dim(\isom(\E,g_i^*\E) \\
             & = &\sum_{i=1}^{n-1}\dim\hom{}{\E,g_i^*\E} \\
	     &  \le & (n-1)n^2
\end{eqnarray*}
where the last inequality follows from \ref{p:ssbound}
since for each $i=1,\dots,n-1$, $\hom{}{\E,g_i^*\E}$ is a 
 semistable bundle  of rank $n^2$ and degree $0$.
\end{proof}

Set $$
\Lambda(n) = 
 (n^3-n^2)  +  \frac{n^2}{4}(g-1)  + \frac{n}{2} + \frac{n^2 g^2}{4} + \frac{1}{4}  
$$
We define a function $h_g:{\mathbb N} \ri {\mathbb N}$
recursively by the formula $h_g(1)=0$
and
$$
h_g(n) - h_g(n-1) = 
\Lambda(n)
$$

\begin{proposition}
 If $n\ge 1$ then $h_g(n+1)\ge h_g(n)$.
\end{proposition}

\begin{theorem}
 We have
\[
 \ed (\bun^{\xi ,ss}_\sln ) \le \lfloor h_g(n)\rfloor + 1 .
\]
\end{theorem}

\begin{proof}
 We induct on $n$. The result for $n=1$ is by choice of the constant.

Let $\E$ be a rank $n\ge 2$ semistable bundle defined over a field $K/\Bbbk$. 
Notice that 
\begin{eqnarray*}
 n^2g^2 - 4(n^2-1)(g-1) &=& n^2 (g-2)^2 + 4g-4\\
			&\ge &0
\end{eqnarray*}
and 
$$
n^3-n^2 = n^2 (n-1) \ge n^2.
$$
The first inequality implies $\frac{n^2g^2}{4} \ge (n^2-1)(g-1) $
so if $\E$ is  stable  then by (\ref{t:bound}) we are done. 

Otherwise 
we can find a Galois
extension $L/K$ with group $G$ and an exact sequence 
\[
 0\ri \K\ri \E\ri V^q\ri 0
\]
defined over $L$. By induction and the above proposition, $\K$ is defined over a field of 
transcendence degree at most $h_g(n-1)$. The stable bundle $V$ is defined over a field of transcendence
degree at most $(\rk(V)^2 -1 )(g-1) + \rk(V)$. Writing $\alpha = \rk(V)$ and applying
\ref{p:definition} to our bundle $\E$, along with its Galois action, descends to a field of of transcendence degree
$$
h_g(n-1) + \alpha^2(g-1) + \alpha - q^2\alpha^2g + \alpha qng + (n^3-n^2)  
$$
It suffices to show that if 
$$
\lambda(\alpha,q) = \alpha^2(g-1) + \alpha - q^2\alpha^2g + \alpha qng
$$
then for all pairs of integers $(\alpha,q)$ with $0<q\alpha <n$ we have
$$
(n^3-n^2)+\lambda(\alpha,q) \le 
\Lambda(n)
$$
(The extra $+1$ is the statement of the theorem comes from the choice of trivialization
$\det\E\cong\xi$.)
To prove the above assertion we consider two cases.

\noindent \underline{Case I: $\alpha\le n/2$ }

In this case, some calculus shows that  
$$
- q^2\alpha^2g + \alpha qng \le n^2g/4,
$$
by considering $p(x)=xng-x^2g$.
So we obtain 
\begin{eqnarray*}
 \lambda(\alpha,q) &\le & \frac{n^2}{4}(g-1) + \frac{n}{2} + \frac{n^2g}{4} 	\\
		& \le & \frac{n^2}{4}(g-1)  + \frac{n}{2} + \frac{n^2 g^2}{4} + \frac{1}{4}.	
\end{eqnarray*}

\noindent \underline{Case II: $\alpha> n/2$ } 

Now we must have $q=1$. The bound is much easier
as 
\begin{eqnarray*}
 \lambda(\alpha,q) &=& -\alpha^2 + \alpha + \alpha ng \\
		& \le & \frac{(1+ng)^2}{4}
\end{eqnarray*}
One obtains
\begin{eqnarray*}
 4.\Lambda(n) - (1+ng)^2 - 4n^3 + 4n^2  &=& n^2(g-1) +2n - 2ng \\
	&=& (g-1)n(n-2).
\end{eqnarray*}
This quantity is non-negative as $n\ge 2$.

\end{proof}

\section{The full moduli stack}

\begin{theorem}
Suppose that ${\rm char}(\Bbbk)=0$.
We have a bound
$$
\ed(\bun^{\xi}_{\sln}) \le \lfloor h_g(n)\rfloor + 1,
$$
where $h_g(n)$ is as defined in the last section.
\end{theorem}
\begin{proof}
We prove this by induction on the rank $n$.
Note that all rank 1 bundles are stable
so that
 $\ed(\bun^{\xi}_{\sln})=\ed(\bun^{s,\xi}_{\sln})$
for $n=1$.
We may assume the result for all $r<n$.
Let $\E$ be an unstable bundle of rank $n$ defined over a field
extension $L/\Bbbk$.  Let $\E'$ be a maximal destabilizing subbundle of $\E'$,
so that we have an exact sequence
\[
 0\ri \E'\ri\E\ri\E/\E'\ri 0
\]
with $\mu(\E')>\mu(\E)$.  By the inductive hypothesis, both
$\E'$ and $\E/\E'$  are defined over smaller fields. Taking the
compositum of these two extensions we  obtain an extension $K$ with
\[\trdeg K \le h_g(\rk(\E')) + h_g(\rk(\E/\E')) .\]
Set $W=\Ext^1(\E/\E',\E')$. The bundle $\E$ is defined over the
function field $K'$ of a subvariety of $\Pr(W)$.
So
\begin{eqnarray*}
 \trdeg K' \le h_g(\rk(\E')) + h_g(\rk(\E/\E'))  + \dim W -1 \\
	   \le h_g(\rk(\E')) + h_g(\rk(\E/\E')) + \rk(\E')(n-\rk(\E'))g -1.
\end{eqnarray*}
Hence it suffices to prove the following inequality:

If $s,t$ are positive integers with $s+t=n$ then
$$h_g(s)+h_g(t) + stg -1\le h_g(n).$$

We may assume that $s\ge t$. We use induction on $t$. For $t=1$
The above inequality turns into
$$h_g(n-1) + (n-1)g \le h_g(n).$$
This follows immediately from the recursive definition of $h_g$
and the fact that $n^3-n^2\ge n^2$ for $n\ge 2$,
as 
$$
                     \frac{n^2g}{4}\ge (n-1)g \qquad \text{when } n\ge 2.
$$

By induction, we may assume that
$$h_g(s)+h_g(t) + stg -1\le h_g(n)$$
and we need to show that
$$h_g(s-1)+h_g(t+1) + (s-1)(t+1)g -1\le h_g(n)$$
provided $s-1\ge t+1$. We may as well prove the following inequality
$$h_g(s-1)+h_g(t+1) + (s-1)(t+1)g -1\le h_g(s)+h_g(t) + stg -1.$$
Rearranging things, we need to show that if $s\ge t+2$ and $t\ge 2$ then
\begin{eqnarray*}
 0 &\le & h_g(s) - h_g(s-1) + h_g(t)- h_g(t+1) + (t-s)g + g \\
   &=& (s^3-s^2) +\frac{s^2}{4}(g-1) + \frac{s}{2} + \frac{s^2g^2}{4}  \\
   & & -((t+1)^3-(t+1)^2) -\frac{(t+1)^2}{4}(g-1) - \frac{t+1}{2} - \frac{(t+1)^2g^2}{4}  \\
   & & +(t-s)g + g.
\end{eqnarray*}
The component functions in the above expression are all increasing, so pairing up like ones we deduce
that it suffices to prove the following
$$
\frac{s^2g^2}{4} - \frac{(t+1)^2g^2}{4} + (t-s)g \ge 0
$$
We write $s=t+\delta$ with $\delta\ge 2$. Multiplying through by $4$ the above becomes
\begin{eqnarray*}
 (t+\delta)^2 g^2 - (t+1)^2 g^2 -4\delta g &= & 
(\delta^2-1)g^2 + 2g(tg(\delta -1) - 2\delta) \\
&\ge& (\delta^2-1)g^2 + 4g(\delta-1)\ge 0
\end{eqnarray*}
This is nonnegative as $t,\delta, g\ge 2$.
\end{proof}

\end{document}